\documentclass[11pt]{article}

\usepackage[T1]{fontenc}
\usepackage{lmodern}
\usepackage[margin=1.15in]{geometry}
\usepackage{amsmath,amssymb,amsthm}
\usepackage{booktabs}
\usepackage{microtype}
\usepackage{array}
\usepackage{longtable}
\usepackage[hidelinks]{hyperref}

\newtheorem{result}{Result}[section]
\newtheorem{theorem}[result]{Theorem}
\newtheorem{lemma}[result]{Lemma}
\newtheorem{proposition}[result]{Proposition}
\newtheorem{corollary}[result]{Corollary}
\theoremstyle{definition}
\newtheorem{definition}[result]{Definition}
\theoremstyle{remark}

\newcommand{\Z}{\mathbb Z}
\newcommand{\F}{\mathbb F}
\newcommand{\Q}{\mathbb Q}
\newcommand{\CW}{CW}
\newcommand{\ICW}{ICW}

\newcommand{\wt}{\operatorname{wt}}

\newcommand{\corr}{\operatorname{cor}}

\title{Character and Multiplier Obstructions for Circulant Weighing Matrices}
\author{Ming Ming Tan\\
School of Computer and Cyber Sciences,\\
Augusta University, Augusta, Georgia, USA}
\date{}

\hypersetup{
  pdftitle={Character and Multiplier Obstructions for Circulant Weighing Matrices},
  pdfauthor={Ming Ming Tan},
  pdfsubject={Character and Multiplier Obstructions for Circulant Weighing Matrices},
  pdfkeywords={circulant weighing matrices, characters of contraction kernels, multipliers, lifting, Eisenstein integers, Gaussian integers, cyclic codes, computer-assisted proof}
}

\begin{document}
\maketitle

\begin{abstract}
We prove the nonexistence of eight circulant weighing matrices from the
remaining table of orders at most $200$ and weights at most $100$.  The
proofs combine contraction, character evaluation on the kernel of a
contraction, multiplier methods, and exact finite computations.  For
$\CW(105,36)$, the contracted matrix is unique up to equivalence.  Applying
a nonprincipal character of the $C_3$ kernel gives an element over the
Eisenstein integers; reduction modulo $1-\omega$ gives a word in a ternary
cyclic code of length $35$, and exact enumeration rules out every required
Eisenstein-unit lift.  For $\CW(140,36)$, the real-valued character
$Y\mapsto-1$ of the $C_4$ kernel is incompatible with the same contracted
class.  For weight $64$, the faithful character $Y\mapsto i$ of a $C_4$
kernel first gives an element of $\Z[i][C_m]$; a generalized multiplier then
forces constancy on multiplication-by-$2$ orbits, and exact correlation
calculations eliminate orders $140$, $180$, and $196$.  The three weight-$49$
cases are settled by the ordinary prime-power multiplier, with contraction
where needed.  Consequently none of
\[
\begin{gathered}
 \CW(105,36),\ \CW(140,36),\ \CW(116,49),\ \CW(120,49),\\
 \CW(192,49),\ \CW(140,64),\ \CW(180,64),\ \CW(196,64)
\end{gathered}
\]
exists.
\end{abstract}



\section{Introduction}

A weighing matrix $W(v,n)$ is a square matrix of order $v$ with entries
$0,\pm1$ satisfying
\[
        WW^T=nI_v.
\]
The integer $n$ is called its weight.  If each row after the first is a cyclic
shift of the preceding row, the matrix is a circulant weighing matrix,
denoted by $\CW(v,n)$.

The existence problem for circulant weighing matrices has been studied using
group rings, character theory, multipliers, and contracted integer weighing
matrices.  An updated form of Strassler's table was given in
\cite{Tan2018}.  Arasu, Gordon, and Zhang subsequently resolved several
entries and listed the remaining parameters in \cite[Table~10]{AGZ2021}.
Version~1.3 of the La Jolla Circulant Weighing Matrix Repository records
$\CW(105,36)$, $\CW(140,36)$, $\CW(120,49)$, $\CW(192,49)$,
$\CW(140,64)$, $\CW(180,64)$, and $\CW(196,64)$ as open
\cite{GordonRepository}.  It records $\CW(116,49)$ as nonexistent by
citing the earlier preprint, a discrepancy discussed below.

Our main result is the following.

\begin{theorem}\label{thm:main}
None of the following circulant weighing matrices exists:
\[
\begin{gathered}
 \CW(105,36),\quad \CW(140,36),\\
 \CW(116,49),\quad \CW(120,49),\quad \CW(192,49),\\
 \CW(140,64),\quad \CW(180,64),\quad \CW(196,64).
\end{gathered}
\]
\end{theorem}

The multiplier-orbit and contraction arguments used in this paper belong to
the established framework of \cite{AGZ2021}; in particular, no new
multiplier theorem is claimed.  The additional step in the weight-$36$ cases
is to apply nonprincipal characters to the kernel of a contraction, thereby
retaining information that is lost when only the fiber sums are recorded.
This gives group-ring equations over cyclotomic integer rings and, for order
$105$, a further reduction to a cyclic code.  Non-lifting arguments have
precedent in \cite{ArasuNabavi2011}, and code-theoretic restrictions on
weighing matrices have precedent in \cite{HaradaMunemasa2012}.

Recent work of Gordon \cite{Gordon2026} also studies lifting in connection
with circulant weighing matrices, but in a different setting: a cyclic
difference set is lifted to a relative difference set, from which a
circulant weighing matrix may be constructed.  Here we begin instead with a
bounded integral contraction of a signed group-ring element and ask whether
it lifts to a $\{0,\pm1\}$-valued element.  The nonprincipal character values
within the fibers provide obstructions to such a lift.  The two approaches
are therefore complementary.

Section~3 describes the relation among the three proof families.  For weight
$36$, the principal character first gives the contracted matrix of order
$35$, and a nonprincipal character of the same contraction kernel is then
used to test whether the contraction lifts.  For weight $64$, the order is
reversed: the faithful character $Y\mapsto i$ of a $C_4$ kernel first gives
an element of $\Z[i][C_m]$, and the generalized multiplier theorem of
\cite{Tan2018} is then applied.  The weight-$49$ proofs remain in the
integral group ring and apply the ordinary prime-power multiplier directly,
possibly after a contraction.

The history of $\CW(116,49)$ requires clarification.  Version~1 of the
arXiv preprint \cite[Proposition~10]{AGZv1} claimed its nonexistence.
However, the displayed equations in that proof have principal sum $9$ and
energy $81$, rather than principal sum $7$ and energy $49$.  The proposition
was not retained in the published version \cite{AGZ2021}, where
$\CW(116,49)$ appears in Table~10 among the remaining open cases.  We give a
complete proof below.

Section~2 fixes notation and records the required multiplier and character
tools.  Section~3 describes the common proof architecture and formulates the
character reductions associated with the contraction kernels.
Section~4 treats the two weight-$36$ cases as
lifting problems over the unique equivalence class of contracted
matrices of order $35$.
Section~5 develops the Gaussian reduction and resolves the three weight-$64$
parameters.  Section~6 treats the weight-$49$ cases by the standard
multiplier-orbit method.  Section~7 describes the exact verification
programs, Section~8 discusses the methodological conclusions, and Section~9
contains the AI usage disclosure.  As in \cite{Tan2018}, previously known
statements are labeled as ``Result.''

\section{Preliminaries}

Let $G$ be a finite abelian group.  The exponent of $G$, denoted by
$\operatorname{exp}(G)$, is the least positive integer $e$ such that $g^e=1$
for every $g\in G$; in particular, $\operatorname{exp}(C_m)=m$.
Let $\zeta_u$ be a primitive $u$th root of unity.  An element of the
cyclotomic group ring has the form
\[
        A=\sum_{g\in G}a_g g,\qquad a_g\in\Z[\zeta_u].
\]
Its augmentation is the sum of its coefficients,
\[
        \epsilon(A)=\sum_{g\in G}a_g.
\]
For integral coefficients and an integer $t$, put
\[
        A^{(t)}=\sum_{g\in G}a_g g^t.
\]
More generally, if $\gcd(t,u|G|)=1$, let
\[
        A^{(t)}=\sum_{g\in G}\sigma_t(a_g)g^t,
        \qquad \sigma_t(\zeta_u)=\zeta_u^t.
\]
For cyclotomic coefficients define
\[
        A^{(-1)}=\sum_{g\in G}\overline{a_g}\,g^{-1}.
\]
These conventions agree with the usual group-ring power and involution when
the coefficients are integral.  A scalar on the right-hand side of a
group-ring equation is understood to multiply the identity element of $G$.

A $G$-invariant weighing matrix of weight $n$ is equivalent to an element
$A\in\Z[G]$ with coefficients in $\{0,\pm1\}$ and
\begin{equation}\label{eq:weighing}
        AA^{(-1)}=n.
\end{equation}
When $G=C_v$, this is a $\CW(v,n)$.  If $n=s^2$, applying the augmentation to
\eqref{eq:weighing} gives
\[
        \epsilon(A)^2=s^2.
\]
After replacing $A$ by $-A$ if necessary, we assume throughout that
\begin{equation}\label{eq:augmentation}
        \epsilon(A)=s.
\end{equation}

Let $v=dm$, write $C_m=\langle X\rangle$, and let
$\rho:C_v\longrightarrow C_m$ be the natural projection.  Its kernel has
order $d$.  If $\gcd(d,m)=1$, we identify
\[
        C_v\cong C_m\times C_d,
\]
in which case $\rho$ forgets the $C_d$-coordinate and its kernel is the
second factor.  If $A$ represents a $\CW(v,n)$, then
\[
        B=\rho(A)=\sum_{j=0}^{m-1}b_jX^j
\]
satisfies
\begin{equation}\label{eq:icw-basic}
        BB^{(-1)}=n,\qquad |b_j|\le d.
\end{equation}
Such an element is called an integer circulant weighing matrix and is denoted
by $\ICW_d(m,n)$.  If $\epsilon(B)=s$, comparison of the augmentation and the
identity coefficient gives
\begin{equation}\label{eq:sum-energy}
        \sum_{j=0}^{m-1}b_j=s,
        \qquad
        \sum_{j=0}^{m-1}b_j^2=n.
\end{equation}
Two elements of $\Z[C_m]$ are called equivalent if one is obtained from the
other by a cyclic translation, an automorphism $X\mapsto X^a$ with
$\gcd(a,m)=1$, and possibly multiplication by $-1$.

For later use, define the periodic correlation
\begin{equation}\label{eq:corr}
        \corr_t(B)=\sum_{j=0}^{m-1}b_{j+t}b_j,
\end{equation}
where subscripts are reduced modulo $m$.  Thus $BB^{(-1)}=n$ if and only if
$\corr_0(B)=n$ and $\corr_t(B)=0$ for $t\ne0$.

\begin{definition}
Let $M\in\Z[C_m]$ and let $t$ be relatively prime to $m$.  The integer $t$ is
a multiplier of $M$ if
\[
        M^{(t)}=gM
\]
for some $g\in C_m$.
\end{definition}

We use the following theorem of McFarland; see
\cite[Theorem~9.1]{McFarland1970} and its group-ring formulation in
\cite[Result~3.2]{Tan2018}.

\begin{result}[McFarland]\label{res:contracted-multiplier}
Let $M$ be an $\ICW_d(m,n)$ with $\gcd(m,n)=1$, and write
$n=\prod_{i=1}^r p_i^{e_i}$.  Suppose that $t$ is relatively prime to $m$ and,
for each $i$, there is an integer $f_i$ such that
\[
        t\equiv p_i^{f_i}\pmod m.
\]
Then $t$ is a multiplier of $M$.  Moreover, $M$ has a translate $M'$ for
which
\[
        {M'}^{(t)}=M'.
\]
\end{result}

For weight a prime power we also use the familiar multiplier result; see
\cite{ArasuSeberry1996,Tan2018,AGZ2021}.

\begin{result}[Prime-power multiplier]\label{res:prime-multiplier}
Let $p$ be prime.  If a $\CW(v,p^{2e})$ exists and $\gcd(v,p)=1$, then a
translate of its group-ring element is fixed by $p$.
\end{result}

We also use the cyclotomic version of the multiplier theorem and its
normalization; see \cite[Theorem~3.4 and Corollary~3.5]{Tan2018}.

\begin{result}[Generalized multiplier theorem]\label{res:generalized-multiplier}
Let $X\in\Z[\zeta_u][G]$, let
$v=\operatorname{lcm}(u,\operatorname{exp}(G))$, and let $t$ satisfy
$\gcd(t,u|G|)=1$.  Suppose
\[
        XX^{(-1)}=n,\qquad \gcd(n,|G|)=1.
\]
If the automorphism $\sigma_t$ fixes every prime ideal of $\Z[\zeta_v]$
lying above a prime divisor of $n$, then
\[
        X^{(t)}=\eta gX
\]
for some $g\in G$ and some root of unity
$\eta=\pm\zeta_u^a$.  If moreover
$\gcd(t-1,\operatorname{exp}(G))=1$, a group translate $Y$ of $X$ may be
chosen so that
\[
        Y^{(t)}=\eta Y.
\]
\end{result}

For a finite abelian group $H$, the \emph{principal character}
$\chi_0$ is defined by $\chi_0(h)=1$ for every $h\in H$; all other
characters are called \emph{nonprincipal}.  A character $\chi$ is
\emph{faithful} if $\ker(\chi)=\{1\}$.  When $H$ is the kernel of a
projection $G\times H\to G$, we shall apply characters of $H$ to retain
information about the coefficients within each fiber of that projection.

The following elementary observation is used throughout.

\begin{lemma}[Character evaluation on a contraction kernel]\label{lem:kernel-character}
Write
\[
 A=\sum_{g\in G}\sum_{h\in H}a_{g,h}(g,h)\in\Z[G\times H]
\]
and suppose $AA^{(-1)}=n$.  For a fixed $g\in G$, call the vector
$(a_{g,h})_{h\in H}$ the $H$-fiber over $g$.  If $\chi$ is a character of
$H$, extend it linearly and put
\[
        A_\chi=(\operatorname{id}\times\chi)(A)
        =\sum_{g\in G}\left(\sum_{h\in H}a_{g,h}\chi(h)\right)g.
\]
Then
\[
        A_\chi A_\chi^{(-1)}=n.
\]
Thus the coefficient of $g$ in $A_\chi$ is the character-weighted sum of the
entries in the $H$-fiber over $g$.
\end{lemma}

\begin{proof}
The character extension is a ring homomorphism and commutes with the
involution.  Applying it to $AA^{(-1)}=n$ proves the assertion.
\end{proof}

\section{Proof architecture and relation to previous methods}

Let $A\in\Z[C_v]$ represent a putative circulant weighing matrix, so that
\[
        AA^{(-1)}=n.
\]
Every proof below replaces this group-ring equation by a finite system of
exact equations in bounded integer or cyclotomic-integer coefficients,
either on a quotient group or on multiplier orbits.  When $v=dm$ with
$\gcd(d,m)=1$, identify $C_v\cong C_m\times C_d$.  Applying the principal
character of the $C_d$ factor records the sum of the coefficients in each
$C_d$-fiber and gives the usual contraction to an $\ICW_d(m,n)$.  Applying
a nonprincipal character of the same factor records a weighted sum of the
entries in each fiber and may distinguish fibers having the same sum.  The
multiplier is applied before or after this character evaluation according to
the available coprimality conditions and the resulting coefficient ring.

Here and below, $\omega$ denotes a primitive cube root of unity.  The proof
chains governing the parameters in Theorem~\ref{thm:main} are
\[
\begin{aligned}
\CW(105,36)
 &\longrightarrow \ICW_3(35,36)
 \longrightarrow \Z[\omega][C_{35}]\\
 &\longrightarrow \F_3[C_{35}]
 \longrightarrow \text{no Eisenstein-unit lift},\\[1mm]
\CW(140,36)
 &\longrightarrow \ICW_4(35,36)
 \longrightarrow \Z[C_{35}]
 \longrightarrow \text{incompatible $Y\mapsto-1$ fiber values},\\[1mm]
\CW(4m,64)
 &\longrightarrow \Z[i][C_m]
 \longrightarrow \text{generalized multiplier}\\
 &\longrightarrow \text{Gaussian correlation obstruction},\\[1mm]
\CW(v,49)
 &\longrightarrow \Z[C_v]\ \text{or an integral contraction}
 \longrightarrow 7\text{-multiplier orbits}\\
 &\longrightarrow \text{integer correlation obstruction}.
\end{aligned}
\]

All four chains begin with the same group-ring equation and end with exact
periodic-correlation conditions, but the reductions retain different
information.  The weight-$49$ arguments are closest to the method of
\cite{AGZ2021}: they use the same multiplier-orbit and contraction framework,
together with coefficient equations arising from $BB^{(-1)}=49$.  The
weight-$36$ arguments start from contracted integer circulant weighing
matrices already isolated there and use nonprincipal characters of the contraction kernels to test the
unresolved lifting problem.  For $\CW(105,36)$, reduction modulo
the prime $1-\omega$ of $\Z[\omega]$ lying above $3$ converts that lifting
problem into a cyclic ternary self-orthogonality condition.  This is related
in spirit to code-based restrictions on weighing matrices
\cite{HaradaMunemasa2012}, although the code arises here only after the
character reduction.  The weight-$64$ chain is a refined application
of the cyclotomic-coefficient and generalized-multiplier framework of
\cite{Tan2018}: the faithful character $Y\mapsto i$ of the $C_4$
factor is applied first, and an exact
Gaussian orbit calculation then completes cases not decided by the broader
arithmetic criteria of that work.  Thus the contribution is not a new
multiplier theorem, but the parameter-specific ordering and combination of
established tools and the resulting exact non-lifting computations.

\subsection{\texorpdfstring{An order-$3$ kernel and an Eisenstein reduction}{An order-3 kernel and an Eisenstein reduction}}

Let $\omega$ be a primitive cube root of unity and put
$\lambda=1-\omega$.  Then
\[
        \lambda\overline{\lambda}=3,
        \qquad
        \Z[\omega]/(\lambda)\cong\F_3.
\]
The six units of $\Z[\omega]$ are
\[
        \{\pm1,\pm\omega,\pm\omega^2\};
\]
they are also the sixth roots of unity.  We refer to them as
\emph{Eisenstein units}.  For
$R=\sum_{g\in G}r_gg\in\F_3[G]$, write
\[
        \wt(R)=|\{g\in G:r_g\ne0\}|
\]
for its Hamming weight.

We first record the general form of the $C_3$-character obstruction.

\begin{theorem}\label{thm:C3-lift}
Let $G$ be a finite abelian group, write $C_3=\langle Y\rangle$, and suppose
$A\in\Z[G\times C_3]$ has coefficients in $\{0,\pm1\}$ and satisfies
$AA^{(-1)}=n$.  Suppose that applying the principal character
$\chi_0(Y)=1$ to the $C_3$ factor gives
\[
        B=3C\in\Z[G],
\]
where $C$ has coefficients in $\{0,\pm1\}$.  Then $9\mid n$ and there exists
$D\in\Z[\omega][G]$ such that
\begin{enumerate}
\item every nonzero coefficient of $D$ is an Eisenstein unit;
\item $D$ has exactly $n/3$ nonzero coefficients;
\item $DD^{(-1)}=n/3$.
\end{enumerate}
Consequently, reducing $D$ modulo $\lambda=1-\omega$ gives
$R\in\F_3[G]$ with
\[
        \wt(R)=n/3,
        \qquad
        RR^{(-1)}=0.
\]
\end{theorem}

\begin{proof}
Write
\[
 A=\sum_{g\in G}g
   \left(a_{g,0}+a_{g,1}Y+a_{g,2}Y^2\right),
 \qquad a_{g,r}\in\{0,\pm1\}.
\]
For $g\in G$, the triple
$(a_{g,0},a_{g,1},a_{g,2})$ is the $C_3$-fiber over $g$, and
\[
        s_g=a_{g,0}+a_{g,1}+a_{g,2}
\]
is its \emph{principal sum}.  Thus $s_g$ is the coefficient of $g$ in the
image of $A$ under the principal character.  A fiber with principal sum
zero will be called a \emph{zero-sum fiber}.

Since character evaluation is a ring homomorphism, $BB^{(-1)}=n$.  The
equality $B=3C$ gives
\[
        CC^{(-1)}=\frac n9.
\]
If $C=\sum_{g\in G}c_gg$, then the coefficient of the identity in
$CC^{(-1)}$ is $\sum_g c_g^2$.  Hence $9\mid n$, and, since
$c_g\in\{0,\pm1\}$, exactly $n/9$ coefficients of $C$ are nonzero.  It
follows that exactly $n/9$ fibers have principal sum $\pm3$.  Such a fiber
is necessarily $(1,1,1)$ or $(-1,-1,-1)$, and these fibers account for
$n/3$ nonzero coefficients of $A$.

The coefficient of the identity in $AA^{(-1)}$ shows that $A$ has exactly
$n$ nonzero coefficients.  The remaining $2n/3$ nonzero coefficients occur
in zero-sum fibers.  A nonzero zero-sum fiber is a permutation of
$(1,-1,0)$ and therefore contains exactly two nonzero entries.  Thus there
are exactly $n/3$ nonzero zero-sum fibers.

Now apply the nonprincipal character $\chi(Y)=\omega$, and write its image as
\[
 E=\sum_{g\in G}
   \left(a_{g,0}+a_{g,1}\omega+a_{g,2}\omega^2\right)g
 \in\Z[\omega][G].
\]
By Lemma~\ref{lem:kernel-character}, $EE^{(-1)}=n$.  A fiber with principal
sum $\pm3$ maps to zero because $1+\omega+\omega^2=0$.  A nonzero zero-sum
fiber maps to a nonzero difference $\omega^r-\omega^s$, where
$r\ne s$, and every such difference has the form
\[
        (1-\omega)u,
        \qquad
        u\in\{\pm1,\pm\omega,\pm\omega^2\}.
\]
It follows that
\[
        E=(1-\omega)D,
\]
where the nonzero coefficients of $D$ are Eisenstein units.  Moreover, a
coefficient of $D$ is nonzero precisely when the corresponding fiber is a
nonzero zero-sum fiber, so $D$ has exactly $n/3$ nonzero coefficients.
Since the involution sends $\omega$ to $\omega^{-1}$,
\[
 n=EE^{(-1)}
  =(1-\omega)(1-\omega^{-1})DD^{(-1)}
  =3DD^{(-1)}.
\]
Therefore $DD^{(-1)}=n/3$.

Finally, $\omega\equiv1\pmod{1-\omega}$, so every Eisenstein unit reduces to
$\pm1$ and remains nonzero.  Thus reduction preserves the support of $D$.
Since $9\mid n$, the scalar $n/3$ is zero in $\F_3$, and reducing
$DD^{(-1)}=n/3$ modulo $1-\omega$ gives $RR^{(-1)}=0$.
\end{proof}

\subsection{\texorpdfstring{The character $Y\mapsto-1$ on a $C_4$ contraction kernel}{The character Y maps to -1 on a C4 contraction kernel}}

The next argument applies the real-valued nonprincipal character
$Y\mapsto-1$ to the $C_4$ factor that is the kernel of the contraction.

\begin{lemma}\label{lem:C4-character}
Let
\[
 A=\sum_{j\in\Z_m}\sum_{r=0}^3 a_{j,r}X^jY^r
 \in\Z[C_m\times C_4]
\]
have coefficients in $\{0,\pm1\}$ and satisfy $AA^{(-1)}=n$.  Put
\[
 b_j=a_{j,0}+a_{j,1}+a_{j,2}+a_{j,3}
\]
and
\[
 f_j=a_{j,0}-a_{j,1}+a_{j,2}-a_{j,3}.
\]
Then
\[
 F=\sum_{j\in\Z_m}f_jX^j
\]
satisfies $FF^{(-1)}=n$ and $|f_j|\le4$.  Moreover, if $b_j=\pm3$, then
$|f_j|=1$.
\end{lemma}

\begin{proof}
The first assertion follows from Lemma~\ref{lem:kernel-character} by applying
the character $Y\mapsto-1$.  If a four-term fiber has sum $3$, it consists
of three $+1$ entries and one zero; if its sum is $-3$, it consists of three
$-1$ entries and one zero.  In either case its alternating character sum is
$\pm1$.
\end{proof}

\section{Weight 36: obstructing lifts from the contraction to order 35}

The two weight-$36$ parameters share the same contracted matrix of order
$35$, but the obstructions to lifting it are different.  For order $105$ we
apply a nonprincipal character of the $C_3$ kernel and then reduce to a
ternary cyclic code.  For order $140$ we apply the real-valued character
$Y\mapsto-1$ of the $C_4$ kernel and obtain a contradiction from the four
fibers supporting the contraction.

\subsection{\texorpdfstring{The contracted matrices of order $35$ and weight $36$}{The contracted matrices of order 35 and weight 36}}

Both $\CW(105,36)$ and $\CW(140,36)$ contract naturally to order $35$.
Write $C_{35}=\langle X\rangle$.  The following exact classification
will be used twice.

\begin{proposition}\label{prop:icw35}
Let $d\in\{3,4\}$.  Every $\ICW_d(35,36)$ of augmentation $6$ is equivalent
to
\begin{equation}\label{eq:B35}
        B=-3+3X^5+3X^{10}+3X^{20}.
\end{equation}
In particular, every nonzero coefficient of such an integer circulant
weighing matrix has absolute value $3$.
\end{proposition}

\begin{proof}
Since
\[
        4\equiv2^2\equiv3^{10}\pmod{35},
\]
Result~\ref{res:contracted-multiplier} gives a translate, still denoted by
$B$, for which $B^{(4)}=B$.

Let $O_0,\ldots,O_8$ be the multiplication-by-$4$ orbits on
$\Z_{35}$, listed by the representatives
\begin{equation}\label{eq:35-reps}
        0,1,2,3,5,6,7,14,15.
\end{equation}
Put $w_i=|O_i|$.  In the order of the representatives in
\eqref{eq:35-reps}, the orbit sizes are
\begin{equation}\label{eq:35-sizes}
        (w_0,\ldots,w_8)=(1,6,6,6,3,6,2,2,3).
\end{equation}

Let $x_i$ be the common coefficient of $B$ on $O_i$.  The augmentation
condition and the coefficient of the identity in $BB^{(-1)}=36$ give the
bounded integer vectors satisfying
\begin{equation}\label{eq:35-sum-energy}
        \sum_{i=0}^8w_ix_i=6,
        \qquad
        \sum_{i=0}^8w_ix_i^2=36.
\end{equation}
We first enumerate the vectors satisfying these two equations and then
impose the nonzero correlation equations $\corr_t(B)=0$.  A \emph{full
solution} is a vector satisfying all these correlations.  The exact counts
are as follows.

\begin{center}
\begin{tabular}{@{}crr@{}}
\toprule
Coefficient bound & Vectors satisfying \eqref{eq:35-sum-energy}
                  & Full solutions\\
\midrule
$3$ & $1{,}434$ & $2$\\
$4$ & $1{,}600$ & $2$\\
\bottomrule
\end{tabular}
\end{center}

With coordinates indexed by the orbit representatives in the order
given in \eqref{eq:35-reps}, the two full solutions are
\begin{equation}\label{eq:35-two}
\begin{split}
   &(-3,0,0,0,0,0,0,0,3),\\
   &(-3,0,0,0,3,0,0,0,0).
\end{split}
\end{equation}
They are interchanged by the automorphism $X\mapsto X^3$.  The second vector
in \eqref{eq:35-two} is precisely \eqref{eq:B35}.  The accompanying program
\texttt{verify\_lifting.cpp} reproduces the orbit data, all counts, and both
solutions using exact integer arithmetic.
\end{proof}

Notice that
\[
        B=3C,
        \qquad
        C=-1+X^5+X^{10}+X^{20},
        \qquad
        CC^{(-1)}=4.
\]
Thus the contracted matrix is three times a $\CW(35,4)$.

\subsection{The Eisenstein and ternary-code obstruction for
\texorpdfstring{$\CW(105,36)$}{CW(105,36)}}

The contraction alone does not determine whether the unique equivalence
class of $\ICW_3(35,36)$ lifts to order $105$.  The $C_3$-character reduction from
Theorem~\ref{thm:C3-lift} supplies the additional obstruction.

\begin{theorem}\label{thm:105}
There is no $\CW(105,36)$.
\end{theorem}

\begin{proof}
Suppose that $A$ is a $\CW(105,36)$.  Identify
\[
        C_{105}\cong C_{35}\times C_3,
        \qquad C_3=\langle Y\rangle.
\]
Applying the principal character $Y\mapsto1$ gives an
$\ICW_3(35,36)$.  By Proposition~\ref{prop:icw35}, this image is equivalent
to $B=3C$, where
\[
        C=-1+X^5+X^{10}+X^{20}.
\]
The translations and the automorphism used in
Proposition~\ref{prop:icw35} may also be applied to $A$.  A translation in
$C_{35}$ has a preimage in $C_{105}$.  The 
automorphism
$X\mapsto X^3$ of $C_{35}$ is induced under the above direct-product
identification by
\[
        (X,Y)\longmapsto(X^3,Y).
\]
Equivalently, if $C_{105}=\langle Z\rangle$, this is the automorphism
$Z\mapsto Z^{73}$, since
\[
        73\equiv3\pmod{35},\qquad 73\equiv1\pmod3,
        \qquad \gcd(73,105)=1.
\]
Thus we may replace $A$ by an equivalent matrix and assume that its image
under the principal character is exactly $B=3C$.

Theorem~\ref{thm:C3-lift} now gives an element
$D\in\Z[\omega][C_{35}]$ with exactly $12$ nonzero coefficients, all
of which are Eisenstein units, and
\begin{equation}\label{eq:D12}
        DD^{(-1)}=12.
\end{equation}
Its reduction $R$ modulo $1-\omega$ has Hamming weight $12$ and satisfies
\begin{equation}\label{eq:Rselforth}
        RR^{(-1)}=0
        \quad\text{in }\F_3[C_{35}].
\end{equation}

We determine all possible $R$.  For a monic polynomial $h$ with nonzero
constant term, write
\[
        h^*(X)=h(0)^{-1}X^{\deg h}h(X^{-1})
\]
for its reciprocal polynomial.  Over $\F_3$,
\begin{equation}\label{eq:factor35}
 X^{35}-1=(X-1)\Phi_5(X)\Phi_7(X)f(X)f^*(X),
\end{equation}
where
\begin{align*}
 f(X)={}&X^{12}+X^{10}-X^8+X^7+X^5-X^4+X^3-X^2-X+1,\\
 f^*(X)={}&X^{12}-X^{11}-X^{10}+X^9-X^8+X^7+X^5-X^4+X^2+1.
\end{align*}
The polynomials $\Phi_5$ and $\Phi_7$ are irreducible and self-reciprocal over
$\F_3$, since $\operatorname{ord}_5(3)=4$ and
$\operatorname{ord}_7(3)=6$, with a power of $3$ equal to $-1$ modulo each
of $5$ and $7$.  The two polynomials $f$ and $f^*$ are reciprocal
irreducible factors of degree $12$: indeed,
$\operatorname{ord}_{35}(3)=12$, and no power of $3$ is $-1$ modulo $35$.

Identify
\[
        \F_3[C_{35}]\cong\F_3[X]/(X^{35}-1),
\]
and write
\[
        r(X)=\sum_{j=0}^{34}r_jX^j
\]
for the unique polynomial of degree less than $35$ representing $R$.  Let
\[
        \tau(r)(X)=r(X^{-1})\pmod{X^{35}-1};
\]
this is the polynomial form of the involution $R\mapsto R^{(-1)}$.  Hence
\eqref{eq:Rselforth} is equivalent to
\[
        r\tau(r)=0
        \quad\text{in }\F_3[X]/(X^{35}-1).
\]
Since $3\nmid35$, the factorization in \eqref{eq:factor35} is square-free,
and the Chinese remainder theorem identifies this quotient with the direct
product of the five corresponding finite fields.  On a component associated
with a self-reciprocal factor, $\tau$ is a field automorphism; the equation
$r\tau(r)=0$ therefore forces that component of $r$ to be zero.  Hence
$X-1$, $\Phi_5$, and $\Phi_7$ divide $r$.

The involution interchanges the two components associated with $f$ and
$f^*$.  If both components of $r$ were nonzero, the corresponding components
of $r\tau(r)$ would also be nonzero.  Thus at least one of $f$ and $f^*$
divides $r$.  Both cannot divide $r$, since together with the three
self-reciprocal factors this would imply $X^{35}-1\mid r$, and hence $r=0$,
contrary to $\wt(R)=12$.  Therefore $R$ lies in one of the two reciprocal
cyclic codes over $\F_3=\{0,\pm1\}$ generated by
\[
        g(X)=(X-1)\Phi_5(X)\Phi_7(X)f(X)
\]
and $g^*(X)$, respectively.  These codes are ternary because their
coefficient field is $\F_3$.  Since $\deg g=23$, each has parameters
$[35,12]$.

The involution $R\mapsto R^{(-1)}$ interchanges the two codes.  It also sends
any Eisenstein-unit lift satisfying \eqref{eq:D12} to another such lift.
It is therefore enough to enumerate the code generated by $g$.  Enumerating
all $3^{12}=531{,}441$ codewords gives the exact weight distribution
\begin{center}
\begin{tabular}{@{}cr@{}}
\toprule
Weight & Number of codewords\\
\midrule
$0$  & $1$\\
$12$ & $420$\\
$15$ & $2{,}520$\\
$18$ & $37{,}590$\\
$21$ & $158{,}550$\\
$24$ & $218{,}610$\\
$27$ & $102{,}620$\\
$30$ & $11{,}130$\\
\bottomrule
\end{tabular}
\end{center}
The entries sum to $3^{12}$.  The $420$ words of weight $12$ form six
classes under cyclic translation and multiplication by $-1$, each of size
$70$; representatives are listed in Appendix~\ref{app:ternary-classes}.
Translation and multiplication by $-1$ preserve the existence of a lift, so
one representative from each class suffices.

For a fixed signed ternary word, each nonzero coefficient has exactly three
Eisenstein-unit lifts with the prescribed reduction modulo $1-\omega$.
A choice of one of these three lifts at each nonzero position will be called
a \emph{phase assignment}.  Multiplication of all coefficients by a common
power of $\omega$ preserves both the ternary reduction and
\eqref{eq:D12}; normalizing one nonzero coefficient therefore leaves
$3^{11}=177{,}147$ phase assignments for each class.  The exact
program tests all
\[
        6\cdot3^{11}=1{,}062{,}882
\]
normalized assignments.  For a sequence of length $35$, the correlation at
shift $35-h$ is the complex conjugate of the correlation at shift $h$;
therefore it is enough to test shifts $1\le h\le17$.  In the order of the
six representatives in Appendix~\ref{app:ternary-classes}, the final
filtering stages are
\begin{center}
\begin{tabular}{@{}crrc@{}}
\toprule
Class & Last shift with survivors & Survivors & Eliminating shift\\
\midrule
$1$ & $6$ & $6$   & $7$\\
$2$ & $6$ & $45$  & $7$\\
$3$ & $4$ & $324$ & $5$\\
$4$ & $8$ & $2$   & $9$\\
$5$ & $5$ & $72$  & $6$\\
$6$ & $8$ & $2$   & $9$\\
\bottomrule
\end{tabular}
\end{center}
Thus no assignment has all nonzero correlations equal to zero.  The complete
stage-by-stage counts are included in the recorded output.  Hence no $D$
satisfying \eqref{eq:D12} exists, contradicting
Theorem~\ref{thm:C3-lift}.  Therefore no $\CW(105,36)$ exists.
\end{proof}

\subsection{The real-character obstruction for
\texorpdfstring{$\CW(140,36)$}{CW(140,36)}}

Here the same contracted class is tested by the real-valued nonprincipal
character $Y\mapsto-1$ of the $C_4$ factor in
$C_{140}\cong C_{35}\times C_4$.

\begin{theorem}\label{thm:140}
There is no $\CW(140,36)$.
\end{theorem}

\begin{proof}
Suppose that $A$ is a $\CW(140,36)$ and identify
$C_{140}\cong C_{35}\times C_4$.  Its image under the principal
character of the $C_4$ factor is an $\ICW_4(35,36)$.  By
Proposition~\ref{prop:icw35}, this element has four nonzero coefficients,
each of absolute value $3$.

Apply the character $C_4\to\{\pm1\}$ from
Lemma~\ref{lem:C4-character}.  Its image $F\in\Z[C_{35}]$ satisfies
$FF^{(-1)}=36$ and has coefficient bound $4$.  Its augmentation is $\pm6$;
after replacing $F$ by $-F$ if necessary, it is an $\ICW_4(35,36)$ of
augmentation $6$.  Proposition~\ref{prop:icw35} therefore says that every
nonzero coefficient of $F$ has absolute value $3$.

On the other hand, at each of the four positions where the
principal-character image has coefficient $\pm3$,
Lemma~\ref{lem:C4-character} gives a coefficient of $F$ of absolute value
$1$.  This contradiction proves the theorem.
\end{proof}

\section{\texorpdfstring{Weight 64: character reduction over $\Z[i]$ and generalized multipliers}{Weight 64: character reduction over Z[i] and generalized multipliers}}

For the weight-$64$ parameters, we first apply the faithful character
$Y\mapsto i$ to the $C_4$ factor that is the kernel of the contraction
$C_{4m}\to C_m$.  This gives an element of $\Z[i][C_m]$, to which the
generalized multiplier theorem is then applied.  Here $\Z[i]$ is the ring of
Gaussian integers, whose units are $\{\pm1,\pm i\}$.  The following
cyclotomic decomposition fact supplies the prime-ideal hypothesis required
by the multiplier theorem.

\begin{lemma}\label{lem:prime2}
Let $m$ be odd, and choose $t$ such that
\[
        t\equiv1\pmod4,
        \qquad
        t\equiv2\pmod m.
\]
Then $\gcd(t,4m)=1$, and the automorphism $\sigma_t$ of
$\Q(\zeta_{4m})$ fixes every prime ideal above $2$.
\end{lemma}

\begin{proof}
Put $K_1=\Q(i)$, $K_2=\Q(\zeta_m)$, and
$K=K_1K_2=\Q(\zeta_{4m})$.  Since the conductors $4$ and $m$ are coprime,
$K_1\cap K_2=\Q$, and restriction gives
\[
 \operatorname{Gal}(K/\Q)
 \cong \operatorname{Gal}(K_1/\Q)\times
       \operatorname{Gal}(K_2/\Q).
\]
The prime $2$ is totally ramified in $K_1$ and unramified in $K_2$.  Thus the
decomposition group above $2$ is the whole group
$\operatorname{Gal}(K_1/\Q)$ on the first factor and is generated by the
Frobenius $\sigma_2$ on the second factor.  Equivalently, under the displayed
product identification it is
\[
 \operatorname{Gal}(K_1/\Q)\times
 \langle\sigma_2|_{K_2}\rangle;
\]
see the decomposition law for primes in cyclotomic fields
\cite[Chapter~2]{Washington1997}.  The congruences defining $t$ show that
$\sigma_t$ restricts to the identity on $K_1$ and to $\sigma_2$ on $K_2$.
Hence $\sigma_t$ lies in the decomposition group of a prime above $2$.
Since $K/\Q$ is abelian, the decomposition groups of all primes above
$2$ coincide; hence $\sigma_t$ fixes every such prime.  Finally, the same
congruences show that $t$ is odd and relatively prime to $m$, so
$\gcd(t,4m)=1$.
\end{proof}

Let $m$ be odd and identify $C_{4m}$ with $C_m\times C_4$.  Write
$C_m=\langle X\rangle$ and $C_4=\langle Y\rangle$.  The faithful
character $Y\mapsto i$ gives the following necessary condition.

\begin{theorem}\label{thm:gaussian-reduction}
If a $\CW(4m,64)$ exists, then there is an element
\[
        E=\sum_{j\in\Z_m}z_jX^j\in\Z[i][C_m]
\]
satisfying
\begin{align}
        EE^{(-1)}&=64,                                      \label{eq:gauss-norm}\\
        z_j&=a_j+ib_j,\quad a_j,b_j\in\{-2,-1,0,1,2\},     \label{eq:gauss-box}\\
        \sum_{j\in\Z_m}z_j&=8,                             \label{eq:gauss-sum}\\
        z_{2j}&=z_j \quad(j\in\Z_m).                       \label{eq:gauss-fixed}
\end{align}
\end{theorem}

\begin{proof}
Suppose
\[
 A=\sum_{j\in\Z_m}\sum_{r=0}^3a_{j,r}X^jY^r
 \in\Z[C_m\times C_4]
\]
represents a $\CW(4m,64)$.  Apply the character $Y\mapsto i$ and put
\[
 E=\sum_{j\in\Z_m}z_jX^j,
 \qquad
 z_j=(a_{j,0}-a_{j,2})+i(a_{j,1}-a_{j,3}).
\]
Lemma~\ref{lem:kernel-character} gives \eqref{eq:gauss-norm}, and the
coefficient formula gives \eqref{eq:gauss-box}.

Choose $t$ by the Chinese remainder theorem as in Lemma~\ref{lem:prime2}.
The lemma shows that $\sigma_t$ fixes every prime ideal above $2$ in
$\Z[\zeta_{4m}]$.  Since $\gcd(64,m)=1$ and $\gcd(t,4m)=1$,
Result~\ref{res:generalized-multiplier} applies to $E$.  Moreover,
$t-1\equiv1\pmod m$, so after a group translation we have
\[
        E^{(t)}=\eta E
\]
for a Gaussian unit $\eta$.

Let $S=\epsilon(E)=\sum_jz_j$.  Applying the principal character to
\eqref{eq:gauss-norm} gives $S\overline S=64$, so $S\ne0$.  Because
$t\equiv1\pmod4$, the automorphism $\sigma_t$ fixes every Gaussian
coefficient, and the augmentation of $E^{(t)}$ is again $S$.  Applying the
augmentation to $E^{(t)}=\eta E$ therefore gives $S=\eta S$, whence
$\eta=1$.  The integer solutions of $a^2+b^2=64$ are
$(\pm8,0)$ and $(0,\pm8)$, so $S$ is a Gaussian unit times $8$.  Multiplying
$E$ by the inverse unit normalizes $S$ to $8$; this rotation preserves the coefficient bounds in
\eqref{eq:gauss-box} and the equation $E^{(t)}=E$.  Finally,
$t\equiv2\pmod m$ turns $E^{(t)}=E$ into \eqref{eq:gauss-fixed}.
\end{proof}

Let $O_1,\ldots,O_s$ be the multiplication-by-$2$ orbits on $\Z_m$, with
sizes $w_1,\ldots,w_s$, and let $c_\ell=r_\ell+is_\ell$ be the common value of
$z_j$ on $O_\ell$.  Theorem~\ref{thm:gaussian-reduction} implies
\begin{align}
 \sum_{\ell=1}^s w_\ell r_\ell&=8,
 &\sum_{\ell=1}^s w_\ell s_\ell&=0,                         \label{eq:gauss-linear}\\
 \sum_{\ell=1}^s w_\ell(r_\ell^2+s_\ell^2)&=64,            \label{eq:gauss-energy}
\end{align}
where $-2\le r_\ell,s_\ell\le2$.  For a shift $h$, put
\[
        \Gamma_h(E)=\sum_{j\in\Z_m}z_{j+h}\overline{z_j}.
\]
The equation \eqref{eq:gauss-norm} requires
$\Gamma_h(E)=0$ for every nonzero shift $h\in\Z_m$.

\begin{lemma}\label{lem:gaussian-enumeration}
For $m=35,45,49$, the bounded integer systems
\eqref{eq:gauss-linear}--\eqref{eq:gauss-energy}, together with the indicated
correlation equations, have the following exact filtering counts.

\begin{center}
\small
\begin{tabular}{@{}clrl@{}}
\toprule
$m$ & Orbit sizes & Initial & Survivors after successive correlations\\
\midrule
$35$ & $1,12,12,3,4,3$ & $1{,}152$
 & $\Gamma_1:4$, $\Gamma_2:4$, $\Gamma_3:4$, $\Gamma_4:4$, $\Gamma_5:0$\\
$45$ & $1,12,4,6,12,4,2,4$ & $58{,}188$
 & $\Gamma_1:1{,}242$, $\Gamma_2:1{,}242$, $\Gamma_3:0$\\
$49$ & $1,21,21,3,3$ & $32$
 & $\Gamma_1:0$\\
\bottomrule
\end{tabular}
\end{center}
In particular, none of the three systems has a full solution, that
is, a solution satisfying all correlation equations.
\end{lemma}

\begin{proof}
The multiplication-by-$2$ orbit representatives are, respectively,
\begin{align*}
 m=35:&\quad 0,1,3,5,7,15,\\
 m=45:&\quad 0,1,3,5,7,9,15,21,\\
 m=49:&\quad 0,1,3,7,21.
\end{align*}
For each modulus, enumerate all pairs
$(r_\ell,s_\ell)\in\{-2,-1,0,1,2\}^2$ satisfying
\eqref{eq:gauss-linear} and \eqref{eq:gauss-energy}, and then impose the
listed Gaussian correlations in order.  Every coefficient in the box
\eqref{eq:gauss-box} is individually realizable as the value of a four-entry
fiber under $Y\mapsto i$.  However, the displayed system records only these
character values; it does not record the principal sum or the number of
nonzero entries in each fiber.  It is therefore a necessary, but not
sufficient, condition for a lift to a $\{0,\pm1\}$-valued element.  In each
case the listed necessary correlation equations already leave no survivor,
so no additional fiber conditions or correlations are required.  The
program \texttt{verify\_lifting.cpp} performs the enumeration with exact
integer arithmetic and checks the displayed orbit data and counts.
\end{proof}

\begin{theorem}\label{thm:weight64}
There are no $\CW(140,64)$, $\CW(180,64)$, or $\CW(196,64)$.
\end{theorem}

\begin{proof}
The three orders are $4m$ with $m=35,45,49$, respectively.  A matrix at any
of these orders would produce the element of $\Z[i][C_m]$ required by
Theorem~\ref{thm:gaussian-reduction}.  Lemma~\ref{lem:gaussian-enumeration}
shows that no such element exists.
\end{proof}

\section{Weight 49: the ordinary multiplier and contraction method}

The weight-$49$ cases form the most classical part of the paper.  Here the
prime-power multiplier acts directly in the integral group ring, or in an
integral contraction of it.  This is the multiplier-orbit framework of
\cite{AGZv1,AGZ2021}; the contribution is the resolution of the three
parameters and the accompanying exact verification data.

\subsection{A family containing \texorpdfstring{$\CW(120,49)$}{CW(120,49)}}

\begin{theorem}\label{thm:20q}
There is no $\CW(20q,49)$ for $1\le q\le6$.
\end{theorem}

\begin{proof}
Suppose a $\CW(20q,49)$ exists, represented by $A\in\Z[C_{20q}]$.
Since $q\le6$, $\gcd(20q,7)=1$.  By
Result~\ref{res:prime-multiplier}, assume $A^{(7)}=A$.  Contract onto $C_{20}$:
\[
        B=\sum_{j=0}^{19}b_jX^j.
\]
Then
\begin{equation}\label{eq:20-basic}
 BB^{(-1)}=49,\qquad \sum_{j=0}^{19}b_j=7,
 \qquad |b_j|\le q\le6,
 \qquad B^{(7)}=B.
\end{equation}
The multiplication-by-$7$ orbits on $\Z_{20}$ are
\begin{align*}
&\{0\},\quad \{10\},\quad \{5,15\},\\
&\{1,3,7,9\},\quad
 \{2,6,14,18\},\quad
 \{4,8,12,16\},\quad
 \{11,13,17,19\}.
\end{align*}
Let their coefficients be
$u,v,d,\alpha,\beta,\gamma,\delta$, respectively.  The augmentation and energy equations, together with the correlations at
shifts $5$ and $10$ after division by their common factor $2$, give
\begin{align}
 u+v+2d+4(\alpha+\beta+\gamma+\delta)&=7,\label{eq:20-sum}\\
 u^2+v^2+2d^2+4(\alpha^2+\beta^2+\gamma^2+\delta^2)&=49,\label{eq:20-energy}\\
 2(\alpha+\delta)(\beta+\gamma)+d(u+v)&=0,\label{eq:20-five}\\
 4\alpha\delta+4\beta\gamma+d^2+uv&=0.\label{eq:20-ten}
\end{align}
Put
\[
 r=u+v,\qquad x=\alpha+\delta,\qquad y=\beta+\gamma.
\]
Equations \eqref{eq:20-sum}, \eqref{eq:20-five}, and the combination of
\eqref{eq:20-energy} with \eqref{eq:20-ten} become
\begin{align}
 r+2d+4x+4y&=7,\label{eq:20-rsum}\\
 dr+2xy&=0,\label{eq:20-rprod}\\
 r^2+4(d^2+x^2+y^2)&=49.\label{eq:20-square}
\end{align}
Equation \eqref{eq:20-rsum} makes $r$ odd, and
\eqref{eq:20-rprod} then makes $d$ even.  Write $d=2e$ and $s=x+y$.
We have
\[
 r=7-4(e+s),\qquad xy=-er.
\]
Substitution into \eqref{eq:20-square} yields
\[
        4s(5s-14)=0,
\]
so $s=0$.  Hence
\[
 r=7-4e,\qquad y=-x,\qquad x^2=e(7-4e).
\]
Since $r^2\le49$, $e\in\{0,1,2,3\}$, and the corresponding right-hand
sides are $0,3,-2,-15$.  Only $e=0$ is possible.  Therefore
\[
 u+v=7,\qquad d=0,\qquad \delta=-\alpha,\qquad\gamma=-\beta.
\]
Equation \eqref{eq:20-ten} now gives
\[
        uv=4(\alpha^2+\beta^2).
\]
Because $|u|,|v|\le6$ and $u+v=7$, we have
$uv\in\{6,10,12\}$.  The first two values are not divisible by $4$; the
last would give $\alpha^2+\beta^2=3$, impossible modulo $4$.  This is a
contradiction.
\end{proof}

\begin{corollary}\label{cor:120}
There is no $\CW(120,49)$.
\end{corollary}

\subsection{The case \texorpdfstring{$\CW(116,49)$}{CW(116,49)}}

\begin{theorem}\label{thm:116}
There is no $\CW(116,49)$.
\end{theorem}

\begin{proof}
Suppose $A$ is a $\CW(116,49)$.  By
Result~\ref{res:prime-multiplier}, assume $A^{(7)}=A$.  The multiplication-by
$7$ orbits on $\Z_{116}$ have representatives
\begin{equation}\label{eq:116-reps}
 0,1,2,3,4,5,6,8,10,15,16,29,30,32,58
\end{equation}
and sizes
\begin{equation}\label{eq:116-sizes}
 1,14,7,14,7,14,7,7,7,14,7,2,7,7,1.
\end{equation}
Let $c_0,\ldots,c_{14}\in\{-1,0,1\}$ be the orbit coefficients and let
$w_i$ denote the corresponding sizes.  Necessarily
\begin{equation}\label{eq:116-sum-energy}
        \sum_iw_ic_i=7,
        \qquad
        \sum_iw_ic_i^2=49.
\end{equation}
An exact enumeration gives the following filtering counts:
\begin{center}
\begin{tabular}{@{}lr@{}}
\toprule
Conditions & Vectors remaining\\
\midrule
\eqref{eq:116-sum-energy} & $6{,}088$\\
Additionally $\corr_1=0$ & $624$\\
Additionally $\corr_2=0$ & $40$\\
Additionally $\corr_3=0$ & $16$\\
Additionally $\corr_4=0$ & $0$\\
\bottomrule
\end{tabular}
\end{center}
For each of the $16$ vectors surviving through $\corr_3=0$, one has
$\corr_4\in\{-1,-3\}$.  They are listed in
Appendix~\ref{app:116}.  Thus no orbit-constant element satisfies
$AA^{(-1)}=49$.
\end{proof}

\subsection{The case \texorpdfstring{$\CW(192,49)$}{CW(192,49)}}

\begin{theorem}\label{thm:192}
There is no $\CW(192,49)$.
\end{theorem}

\begin{proof}
Suppose $A$ is a $\CW(192,49)$.  By
Result~\ref{res:prime-multiplier}, assume $A^{(7)}=A$.  Contract onto $C_{64}$.
The image $B$ is an $\ICW_3(64,49)$ fixed by $7$.  The orbits have
representatives
\begin{equation}\label{eq:64-reps}
 0,1,2,3,4,6,8,9,11,12,16,18,22,24,32,36,44
\end{equation}
and sizes
\begin{equation}\label{eq:64-sizes}
 1,8,4,8,2,4,2,8,8,2,2,4,4,2,1,2,2.
\end{equation}
Let the orbit coefficients be $c_0,\ldots,c_{16}\in\{-3,\ldots,3\}$,
and let $w_i$ be the corresponding orbit sizes.  The augmentation and energy
equations are
\[
        \sum_{i=0}^{16}w_ic_i=7,
        \qquad
        \sum_{i=0}^{16}w_ic_i^2=49.
\]
They leave $22{,}880{,}810$ vectors.  The successive exact filtering counts
are
\begin{center}
\begin{tabular}{@{}lr@{}}
\toprule
Last correlation imposed & Vectors remaining\\
\midrule
none & $22{,}880{,}810$\\
$\corr_1$ & $5{,}952{,}866$\\
$\corr_2$ & $455{,}924$\\
$\corr_3$ & $413{,}640$\\
$\corr_4$ & $14{,}670$\\
$\corr_5$ & $14{,}670$\\
$\corr_6$ & $7{,}836$\\
$\corr_7$ & $7{,}836$\\
$\corr_8$ & $338$\\
$\corr_{11}$ & $338$\\
$\corr_{12}$ & $22$\\
$\corr_{16}$ & $4$\\
\bottomrule
\end{tabular}
\end{center}
In the orbit order \eqref{eq:64-reps}, the four remaining vectors are
\begin{align*}
&(-1,0,0,0,-1,0, 3,0,0, 2,2,0,0,-1,0, 1,-2),\\
&(-1,0,0,0, 1,0, 3,0,0,-2,2,0,0,-1,0,-1, 2),\\
&(0,0,0,0,-1,0,-1,0,0, 2,2,0,0, 3,-1, 1,-2),\\
&(0,0,0,0, 1,0,-1,0,0,-2,2,0,0, 3,-1,-1, 2).
\end{align*}
Each has $\corr_{24}=12$, rather than zero.  Hence no projected matrix, and
therefore no $\CW(192,49)$, exists.
\end{proof}

\section{Exact computation and reproducibility}

Two self-contained C++17 programs accompany the paper.

\begin{itemize}
\item \texttt{verify\_lifting.cpp} reproduces Proposition~\ref{prop:icw35},
checks the factorization \eqref{eq:factor35}, enumerates all $3^{12}$ words of
the ternary $[35,12]$ code, partitions its $420$ weight-$12$ words into the
six translation-and-sign classes, tests all $1{,}062{,}882$ normalized
Eisenstein-unit lifts, verifies the local order-$4$ character calculation in
Theorem~\ref{thm:140}, and reproduces all Gaussian orbit counts in
Lemma~\ref{lem:gaussian-enumeration}.
\item \texttt{verify\_weight49.cpp} reproduces the orbit data and every exact
filtering count used for Theorems~\ref{thm:20q}, \ref{thm:116}, and
\ref{thm:192}.
\end{itemize}

The programs use integer arithmetic only.  No floating-point comparison,
randomized search, or heuristic acceptance criterion occurs.  Pruning is
limited to exact feasibility tests for the remaining augmentation and energy.
For the Eisenstein search, the program checks shifts $1$ through $17$;
shifts $18$ through $34$ are their complex conjugates.  The reciprocal
ternary code is covered by the involution, as explained in the proof of
Theorem~\ref{thm:105}.  For the Gaussian cases, the listed correlations are necessary conditions,
but they already eliminate every vector in the bounded search space defined
by \eqref{eq:gauss-box}; this space contains every character image arising
from a four-entry fiber.

The recorded outputs were reproduced with GCC~14.2.0 using
\begin{verbatim}
g++ -O3 -std=c++17 verify_lifting.cpp -o verify_lifting
g++ -O3 -std=c++17 verify_weight49.cpp -o verify_weight49
./verify_lifting
./verify_weight49
\end{verbatim}
The source files and a README file are included in \url{https://github.com/mmtan/cw-verification}.  The README specifies the exact
representations used for Eisenstein and Gaussian integers, the integer types,
and elementary bounds excluding overflow.  The files specify the complete
finite search spaces, all filtering counts, the final surviving vectors, and
the nonzero correlations that eliminate them. 

\section{Discussion and conclusion}

\begin{proof}[Proof of Theorem~\ref{thm:main}]
The nonexistence of $\CW(105,36)$ and $\CW(140,36)$ follows from
Theorems~\ref{thm:105} and \ref{thm:140}.  The three weight-$64$ cases follow
from Theorem~\ref{thm:weight64}.  The weight-$49$ cases follow from
Corollary~\ref{cor:120} and Theorems~\ref{thm:116} and \ref{thm:192}.
\end{proof}

The weight-$49$ proofs use the established multiplier-orbit and contraction
method directly.  The additional ingredient in the weight-$36$ and
weight-$64$ cases is character evaluation on the kernel of a contraction.
The principal character records only the sum within each fiber, whereas a
nonprincipal character records a weighted sum and can distinguish fibers
having the same principal sum.  For order $105$ this produces an
Eisenstein-integer equation and, after reduction modulo $1-\omega$, a
ternary cyclic-code obstruction.  For order $140$ the character
$Y\mapsto-1$ gives an immediate incompatibility with the classified
contraction.  For weight $64$ the faithful character $Y\mapsto i$ is applied
before the generalized multiplier, reducing the problem to exact bounded
correlation systems over $\Z[i]$.

The use of lifting here is complementary to the relative-difference-set
lifting studied in \cite{Gordon2026}.  In that setting a difference set is
lifted to a relative difference set and then used to construct a circulant
weighing matrix.  Here the starting point is a surviving integer contraction
of a signed group-ring element, and the objective is to prove that no
$\{0,\pm1\}$-valued lift exists.  The character values of the contraction
kernel provide the missing fiber information.  This approach may be useful
for other parameters for which contraction leaves only a small number of
candidate integer circulant weighing matrices.

\section{AI usage disclosure}

\textbf{Use of generative AI.}
During the development of this work, OpenAI GPT-5.6 Pro was used interactively to explore proof approaches, identify gaps in preliminary arguments, assist in the development and checking of exact verification programs, and improve the exposition. All mathematical arguments and computational results reported in the paper were subsequently checked by the author, who takes full responsibility for the content.

\appendix

\section{The six ternary equivalence classes}\label{app:ternary-classes}

A sign $+$ at position $j$ means coefficient $1\in\F_3$, and a sign $-$ means
coefficient $-1\in\F_3$.  The six representatives, under cyclic translation
and negation, are
\begin{align*}
&9^+,13^+,14^-,15^-,19^+,20^+,22^+,23^-,27^-,28^+,33^-,34^-,\\
&7^+,10^+,12^+,14^-,16^+,17^-,19^-,24^+,27^-,30^-,31^-,34^+,\\
&6^+,13^-,14^+,18^+,19^-,21^-,24^+,26^+,27^+,31^-,32^-,34^-,\\
&6^+,8^+,12^+,13^+,18^-,22^-,24^+,27^-,31^-,32^+,33^-,34^-,\\
&6^+,8^-,12^-,14^-,19^+,21^+,22^+,24^+,26^-,31^-,33^+,34^-,\\
&6^+,7^+,10^-,11^-,14^-,16^-,20^-,25^+,27^-,30^+,31^+,34^+.
\end{align*}
For each line, the program tests all $3^{11}$ normalized assignments of
powers of $\omega$ and finds no solution of $DD^{(-1)}=12$.

\section{Final orbit vectors for \texorpdfstring{$\CW(116,49)$}{CW(116,49)}}\label{app:116}

The following are all vectors satisfying \eqref{eq:116-sum-energy} and
$\corr_1=\corr_2=\corr_3=0$, in the orbit order
\eqref{eq:116-reps}.  The final column is $\corr_4$.

\begin{center}
\scriptsize
\renewcommand{\arraystretch}{1.06}
\begin{tabular}{@{}>{\ttfamily}l r@{}}
\toprule
\multicolumn{1}{c}{Orbit vector} & $\corr_4$\\
\midrule
0 0 -1 0 -1 0 0 1 1 0 -1 0 1 1 0 & -3\\
0 0 -1 0 -1 0 1 1 1 0 -1 0 0 1 0 & -1\\
0 0 -1 0 0 0 1 -1 -1 0 1 0 1 1 0 & -1\\
0 0 -1 0 0 0 1 1 -1 0 1 0 1 -1 0 & -3\\
0 0 -1 0 1 0 0 -1 1 0 1 0 1 -1 0 & -3\\
0 0 -1 0 1 0 1 -1 -1 0 0 0 1 1 0 & -3\\
0 0 -1 0 1 0 1 -1 1 0 1 0 0 -1 0 & -1\\
0 0 -1 0 1 0 1 1 -1 0 0 0 1 -1 0 & -1\\
0 0 1 0 -1 0 0 1 -1 0 -1 0 1 1 0 & -1\\
0 0 1 0 -1 0 1 1 -1 0 -1 0 0 1 0 & -3\\
0 0 1 0 0 0 -1 -1 1 0 1 0 -1 1 0 & -1\\
0 0 1 0 0 0 -1 1 1 0 1 0 -1 -1 0 & -3\\
0 0 1 0 1 0 -1 -1 1 0 0 0 -1 1 0 & -3\\
0 0 1 0 1 0 -1 1 1 0 0 0 -1 -1 0 & -1\\
0 0 1 0 1 0 0 -1 -1 0 1 0 1 -1 0 & -1\\
0 0 1 0 1 0 1 -1 -1 0 1 0 0 -1 0 & -3\\
\bottomrule
\end{tabular}
\end{center}

\end{document}